\documentclass[12pt]{amsart}
 
\usepackage{amsmath, amsfonts, amssymb, amsthm, cite, amscd, verbatim, dsfont}

\numberwithin{equation}{section}

\makeatletter
\newtheorem*{rep@theorem}{\rep@title}
\newcommand{\newreptheorem}[2]{%
\newenvironment{rep#1}[1]{%
 \def\rep@title{#2 \ref{##1}}%
 \begin{rep@theorem}}%
 {\end{rep@theorem}}}
 \makeatother

\theoremstyle{plain}
\newtheorem{theorem}{Theorem}[section]
\newreptheorem{theorem}{Theorem}
\newtheorem{lemma}[theorem]{Lemma}

\newtheorem{proposition}[theorem]{Proposition}
\theoremstyle{definition}
\newtheorem{definition}[theorem]{Definition}

\theoremstyle{remark}
\newtheorem*{remark}{Remark}

\newcommand{\Sp}{\operatorname{Sp}}
\newcommand{\SO}{\operatorname{SO}}

\begin{document}

\title[$U(\Omega,m)$ sets of a hyperelliptic curve]{A characterization of the $U(\Omega,m)$ sets of a hyperelliptic curve as $\Omega$ and $m$ vary}
\author{Christelle Vincent}
\address{Department of Mathematics and Statistics, University of Vermont, 16 Colchester Avenue, Burlington VT 05401}
\email{christelle.vincent@uvm.edu}

\begin{abstract}
In this article we consider a certain distinguished set $U(\Omega,m) \subseteq \{1,2,\ldots,2g+1,\infty\}$ that can be attached to a marked hyperelliptic curve of genus $g$ equipped with a small period matrix $\Omega$ for its polarized Jacobian. We show that as $\Omega$ and the marking $m$ vary, this set ranges over all possibilities prescribed by an argument of Poor.
\end{abstract}

\maketitle

\section{Introduction and Statement of Results}

Let $X$ be a hyperelliptic curve of genus $g$ defined over $\mathbb{C}$, and let $J(X)$ be its polarized Jacobian. In Definition \ref{def:periodmatrix}, we associate to $J(X)$ a \emph{small period matrix} $\Omega$, which is an element of the Siegel upper half-space $\mathbb{H}_g$ with the property that there is an isomorphism
\begin{equation}
J(X)(\mathbb{C}) \cong \mathbb{C}^g/L_\Omega,
\end{equation}
where $L_\Omega$ is the rank $2g$ lattice generated by the columns of $\Omega$ and the standard basis $\{e_i\}$ of $\mathbb{C}^g$.

After this choice we may define an analytic theta function 
\begin{equation}
\vartheta(z,\Omega) \colon \mathbb{C}^g \to \mathbb{C},
\end{equation}
whose exact definition is given in Definition \ref{def:thetafunction}. While this function is not well-defined on $J(X)(\mathbb{C})$, it is quasi-periodic with respect to the lattice $L_\Omega$, and so its zero set on the Jacobian is well-defined. In this article we study how a certain combinatorial characterization of this zero set depends on the choice of small period matrix (since the theta function itself depends on the small period matrix) and on a further choice we now make.

Since $X$ is hyperelliptic, there is a morphism $\pi \colon X \to \mathbb{P}^1$ of degree two, branched at $2g+2$ points. Suppose further that $X$ is given a \emph{marking of its branch points}, denoted $m$, by which we mean that the branch points of $\pi$ are numbered $1, 2,\ldots, 2g+1, \infty$. As we explain in Proposition \ref{prop:GBisom}, this choice gives a bijection between sets
\begin{equation}
S \subseteq \{1, 2,\ldots, 2g+1, \infty\}, \quad \# S \equiv 0 \pmod{2},
\end{equation}
up to the equivalence $S \sim S^c$, where $^c$ denotes taking the complement within $\{1, 2,\ldots, 2g+1, \infty\}$, and the two-torsion in $J(X)(\mathbb{C})$.

Then we have the following theorem, which we will repeat and make more precise in Theorem \ref{thm:vanishing}:
\begin{theorem}[Riemann Vanishing Theorem]
Let $X$ be a hyperelliptic curve, $m$ be a marking of its branch points, and let $\Omega$ be a small period matrix associated to its polarized Jacobian. Then there is a distinguished set $\Theta$ on $J(X)(\mathbb{C})$ (defined in Definition \ref{def:Thetaset}) and the zero set of the theta function $\vartheta(z,\Omega)$, considered as a subset of $J(X)(\mathbb{C})$, is exactly the set $\Theta$ translated by an element of the two-torsion of $J(X)$. 

Under the correspondence given above, this two-torsion point corresponds to a set which we denote $T(\Omega,m)$. Note that the set $T(\Omega,m)$ is only well-defined up to the equivalence $S \sim S^c$, where as before $^c$ denotes the complement.
\end{theorem}

This theorem gives rise to the following distinguished set:

\begin{definition}\label{def:Uset}
Let $X$ be a hyperelliptic curve of genus $g$, $\Omega$ a choice of small period matrix associated to its Jacobian via the process described in Definition \ref{def:periodmatrix}, and $m$ a marking of the branch points of $X$. Let $U(\Omega,m) \subset \{1, 2,\ldots, 2g+1, \infty\}$ be defined up to the equivalence $S \sim S^c$ by the following formula:
\begin{equation}\label{eq:Uset}
U(\Omega,m) = 
\begin{cases}
T(\Omega,m) & \text{if $g$ is odd, and} \\
T(\Omega,m) \circ \{\infty\} & \text{if $g$ is even,}
\end{cases}
\end{equation}
where $\circ$ here denotes the symmetric difference of sets (see Definition \ref{def:symmetricdifference}). To fix one set in this equivalence class, we take $U(\Omega,m)$ to be the set containing $\infty$.
\end{definition}

\begin{remark}
We note that Mumford \cite{mumford2} adopts the opposite convention and chooses $U(\Omega,m)$ to be the member of the equivalence class that does not contain $\infty$. In this respect we follow the convention adopted by Poor \cite{poor}.
\end{remark}

The significance of this set $U(\Omega,m)$ is especially salient in computational applications; we invite the reader to consult Section \ref{sec:vanishing} for a further account of its role. This set first appeared in work of Mumford \cite{mumford2}, where given a marked hyperelliptic curve $X$, the author constructs a certain small period matrix $\Omega$ and computes the set $U(\Omega,m)$ explicitly. In this example, it is the case that 
\begin{equation}
\# U(\Omega,m) = g+1,
\end{equation}
where as before $g$ is the genus of the curve. In the theorems following this computation (in particular Mumford's version of Theorem \ref{thm:vanishingtheorem}, Theorem 9.1 of \cite{mumford2}, which is the most important of those from our point of view), the set $U(\Omega,m)$ is always assumed to have this cardinality.

However, in later work of Poor \cite{poor}, the same set $U(\Omega,m)$ is shown to have the property that
\begin{equation}
\# U(\Omega,m) \equiv g+1 \pmod{4}
\end{equation}
(see \cite[Proposition 1.4.9]{poor}). This raises the following interesting question: Does the set $U(\Omega,m)$ always have cardinality $g+1$, or do other cardinalities occur? We answer this question completely:

\begin{theorem}\label{thm:main}
Let $g \geq 1$ and $X$ be a hyperelliptic curve of genus $g$ defined over $\mathbb{C}$. Then for any set $U \subseteq  \{1, 2,\ldots, 2g+1, \infty\}$ containing $\infty$ such that 
\begin{equation}
\# U \equiv g+1 \pmod{4},
\end{equation}
there exists a small period matrix $\Omega$ associated to the Jacobian of $X$ via the process described in Definition \ref{def:periodmatrix}, and a marking $m$ of the branch points of $X$ such that 
\begin{equation}
U = U(\Omega,m).
\end{equation}
\end{theorem}

In other words, every possible set $U$ occurs as the set $U(\Omega,m)$ for a given hyperelliptic curve $X$, and Poor's characterization of $U(\Omega,m)$ is sharp.

\section*{Acknowledgments}

The author would like to thank Sorina Ionica for many discussions that led her to understand the set $U(\Omega,m)$ more deeply, and the referees for their careful reading of the manuscript and thoughtful comments.

\section{Preliminaries}\label{sec:preliminaries}

Let $X$ be a hyperelliptic curve, by which we mean a smooth complete curve of genus $g$ defined over $\mathbb{C}$ admitting a map $\pi \colon X \to \mathbb{P}^1$ of degree $2$. Throughout we denote its Jacobian variety by $J(X)$.

\subsection{The small period matrix of the Jacobian of a curve}\label{sec:periodmatrix}

We give here standard facts about abelian varieties and Jacobians. We refer the reader to \cite{Birkenhake} for further background and proofs.

We begin by giving an analytic space associated to polarized abelian varieties of dimension $g$:
\begin{definition}\label{def:siegelspace}
Let $g \geq 1$. The \emph{Siegel upper half-space} $\mathbb{H}_g$ is the set of symmetric $g \times g$ complex matrices $M$ such that the imaginary part of $M$ (obtained by taking the imaginary part of each entry in $M$) is positive definite. 
\end{definition}

Although much of the discussion below would apply to general polarized abelian varieties, in this article we focus our attention to Jacobians of curves equipped with their principal polarization. To simplify matters, at this time we restrict our attention to these objects. In this setting, the connection between this space and Jacobians is through the following object:

\begin{definition}\label{def:periodmatrix}
Let $X$ be a curve of genus $g$ defined over $\mathbb{C}$, and let $J(X)$ be its principally polarized Jacobian. To $J(X)$, we can associate matrices $\Omega \in \mathbb{H}_g$ in the following manner: Let $A_i$, $B_i$, $i = 1, \ldots, g$, be a basis for the homology group $H_1(J(X),\mathbb{Z})\cong H_1(X,\mathbb{Z})$, which is a $2g$-dimensional vector space over $\mathbb{C}$. Assume further that this basis is symplectic with respect to the cup product. There exists a unique basis $\omega_1, \omega_2, \ldots, \omega_g$ of $\Omega^1(J(X)) \cong \Omega^1(X)$, the space of holomorphic $1$-forms on $J(X)$ or $X$, such that
 \begin{equation}
\int_{B_i} \omega_j = \delta_{ij},
\end{equation}
where $\delta_{ij}$ is the Kronecker delta function. Then the matrix given by $\int_{A_i} \omega_j$ belongs to $\mathbb{H}_g$ and is called a \emph{small period matrix} for $J(X)$.
\end{definition}

Let $\Sp_{2g}(\mathbb{Z})$ be the group of $2g \times 2g$ matrices with coefficients in $\mathbb{Z}$ and symplectic with respect to the bilinear form given by the matrix
\begin{equation}
\begin{pmatrix}
0 & \mathds{1}_g \\ -\mathds{1}_g & 0
\end{pmatrix},
\end{equation}
where $\mathds{1}_g$ is the $g \times g$ identity matrix. We note that two elements of $\mathbb{H}_g$ can be associated to isomorphic polarized abelian varieties if and only if they differ by a matrix in $\Sp_{2g}(\mathbb{Z})$, where the action of $\Sp_{2g}(\mathbb{Z})$ on $\mathbb{H}_g$ is given in the following manner: Let 
\begin{equation}
\gamma = \begin{pmatrix} A & B \\ C & D \end{pmatrix} \in \Sp_{2g}(\mathbb{Z}),
\end{equation}
where $A$, $B$, $C$ and $D$ are four $g \times g$ matrices. Then
\begin{equation}\label{def:sympaction}
\gamma \cdot \Omega =  (A\Omega +B)(C \Omega + D)^{-1},
\end{equation}
where on the right multiplication and addition are the usual operation on $g \times g$ matrices.

We can further define an Abel-Jacobi map for a principally polarized Jacobian variety $J(X)$:
\begin{definition}\label{def:abeljacobi}
Let $X$ be a curve of genus $g$ defined over $\mathbb{C}$, let $J(X)$ be its principally polarized Jacobian, and fix $A_i$, $B_i$, $i = 1, \ldots, g$, a symplectic basis for the homology group $H_1(X,\mathbb{Z})$. Let $\omega_1, \omega_2, \ldots, \omega_g$ be the basis of $\Omega^1(X)$ described in Definition \ref{def:periodmatrix}, $\Omega$ be the small period matrix attached to $J(X)$ via this choice of symplectic basis for homology and let $L_\Omega$ be the rank $2g$ lattice generated by the columns of $\Omega$ and the standard basis $\{e_i\}$ of $\mathbb{C}^g$. Then there is an isomorphism called the \emph{Abel-Jacobi map}
\begin{equation}
AJ \colon J(X)  \to \mathbb{C}^g/L_\Omega,
\end{equation}
given by the map
\begin{equation}
D = \sum_{k=1}^s P_k - \sum_{k=1}^s Q_k  \mapsto \left(\sum_{k=1}^s \int_{Q_k}^{P_k} \omega_i\right)_i,
\end{equation}
where the $P_k$s and $Q_k$s are points on $X$. This map is well-defined since the value of each integral on $X$ is well-defined up to the value of integrating the differentials $\omega_i$ along the basis elements $A_i$, $B_i$, and thus up to elements of $L_\Omega$.
\end{definition}

We will in fact need a slightly modified version of this Abel-Jacobi map for our purposes:
\begin{definition}\label{def:abeljacobimod}
Let $X$ be a curve of genus $g$ defined over $\mathbb{C}$, let $J(X)$ be its principally polarized Jacobian, and fix $A_i$, $B_i$, $i = 1, \ldots, g$, a symplectic basis for the homology group $H_1(X,\mathbb{Z})$. Let $\Omega$ be the small period matrix attached to $J(X)$ via this choice of symplectic basis for homology and let $L_\Omega$ be the rank $2g$ lattice generated by the columns of $\Omega$ and the standard basis $\{e_i\}$ of $\mathbb{C}^g$. This gives rise to an isomorphism
\begin{equation}
\mathbb{C}^g/L_\Omega \to \mathbb{R}^{2g}/\mathbb{Z}^{2g},
\end{equation}
given by writing an element of $\mathbb{C}^g/L_\Omega$ as a linear combination of the columns of $\Omega$ and the standard basis $\{e_i\}$ of $\mathbb{C}^g$ and sending the element to the coefficients of the linear combination. Composing this isomorphism with the Abel-Jacobi map defined in Definition \ref{def:abeljacobi}, we obtain the \emph{modified Abel-Jacobi map}
\begin{equation}
AJ_c \colon J(X) \to \mathbb{R}^{2g}/\mathbb{Z}^{2g},
\end{equation}
which gives the coordinates of a point of $J(X)$ under the Abel-Jacobi map.
\end{definition}

In this paper we will need to know how a change of symplectic basis for $H_1(X,\mathbb{Z})$ affects the image of the Abel-Jacobi map and the coordinates of a point of $J(X)$ under the Abel-Jacobi map. We have
\begin{proposition}[adapted from Section 1.4 of \cite{poor}]\label{prop:sympaction}
Let $X$ be a curve of genus $g$ defined over $\mathbb{C}$, let $J(X)$ be its principally polarized Jacobian, and let $A_i$, $B_i$ be a symplectic basis for $H_1(X,\mathbb{Z})$ from which arises the small period matrix $\Omega$, the Abel-Jacobi map $AJ$ and the modified Abel-Jacobi map $AJ_c$. Let $\gamma \in \Sp_{2g}(\mathbb{Z})$ act on the elements $A_i$, $B_i$. Since $\Sp_{2g}(\mathbb{Z})$ preserves the cup pairing, the images $\tilde{A}_i$, $\tilde{B}_i$ gives rise to a second Abel-Jacobi map $\widetilde{AJ}$. If 
\begin{equation}
\gamma = \begin{pmatrix} A & B \\ C & D \end{pmatrix},
\end{equation}
where $A$, $B$, $C$ and $D$ are $g \times g$ matrices, then
\begin{equation}
\widetilde{AJ} = (C\Omega+ D)^{-T} AJ,
\end{equation}
where $M^{-T}$ is the inverse of the transpose of the matrix $M$. Furthermore, we have
\begin{equation}
\widetilde{AJ}_c = \gamma^{-T} AJ_c.
\end{equation}
\end{proposition}

\subsection{The two-torsion on the Jacobian of a hyperelliptic curve}

We now turn our attention to the two-torsion of the Jacobian of a hyperelliptic curve of genus $g$ defined over $\mathbb{C}$. As a group, it is isomorphic to $C_2^{2g}$, where $C_2$ is the cyclic group with two elements.

Throughout, let $B = \{1,2, \ldots, 2g+1,\infty\}$. When $S \subseteq B$, we let $S^c$ be the complement of $S$ in $B$.

\begin{definition}\label{def:symmetricdifference}
Let $S_1$ and $S_2$ be any two subsets of $B$. We define
\begin{equation}
S_1 \circ S_2 = (S_1 \cup S_2) - (S_1 \cap S_2),
\end{equation}
the \emph{symmetric difference} of $S_1$ and $S_2$.
\end{definition}

This binary operation on subsets in turns gives rise to the following group:
\begin{proposition}\label{prop:defGB}
The set 
\begin{equation}
\{S \subseteq B : \# S \equiv 0 \pmod{2} \} / \{S \sim S^c\}
\end{equation}
is a commutative group under the operation $\circ$, of order $2^{2g}$, with identity $\emptyset \sim B$. Since $S \circ S = \emptyset$ for all $S \subseteq B$, this is a group of exponent $2$. Therefore this group, which we denote $G_B$, is isomorphic to $C_2^{2g}$.
\end{proposition}

If the hyperelliptic curve $X$ is equipped with a marking of its branch points (recall that this means that we label the $2g+2$ branch points of the degree two map $\pi \colon X \to \mathbb{P}^1$, $P_1, P_2, \ldots, P_{2g+1}, P_\infty$), there is in fact an explicit isomorphism between $G_B$ and $J(X)[2]$, the two-torsion on the Jacobian of $X$:

\begin{proposition}[Corollary 2.11 of \cite{mumford2}]\label{prop:GBisom}
To each set $S \subseteq B$ such that $\# S \equiv 0 \pmod{2}$, associate the divisor class of the divisor
\begin{equation}\label{eq:2torsion}
e_S = \sum_{i \in S} P_i - (\#S) P_{\infty}.
\end{equation}
This association is a group isomorphism between $J(X)[2]$ and $G_B$.
\end{proposition}

We may now compose the isomorphism of Proposition \ref{prop:GBisom} with the modified Abel-Jacobi map given in Definition \ref{def:abeljacobimod}.
\begin{definition}\label{def:etamap}
We denote by $\eta_{\Omega,m}$ the isomorphism
\begin{equation}
\eta_{\Omega,m} \colon \{S \subseteq B : \# S \equiv 0 \pmod{2} \} / \{S \sim S^c\} \to (\frac{1}{2}\mathbb{Z})^{2g}/\mathbb{Z}^{2g}
\end{equation}
given by composing the isomorphism $G_B \to J(X)[2]$ given in Proposition \ref{prop:GBisom} and the map $AJ_c$ given in Definition \ref{def:abeljacobimod}.
\end{definition}

\begin{remark}
We note that in Poor's work \cite{poor}, this is the \emph{class} of the map $\eta$, which is an equivalence class of maps to $(\frac{1}{2}\mathbb{Z})^{2g}$. In this work we will not need the distinction between the ``true" $\eta$-map and its class, and therefore by a slight abuse of notation we consider the map above to be the $\eta$-map.
\end{remark}

This map $\eta_{\Omega,m}$ will allow us to give a more concrete definition of the set $U(\Omega,m)$, which we will use in our proof in Section \ref{sec:proof}. We first need one more notion.

\begin{definition}\label{def:estar}
If $x \in \mathbb{C}^{2g}$, let $x = (x_1, x_2)$, with $x_i \in \mathbb{C}^g$; in other words let $x_1$ denote the vector of the first $g$ entries of $x$, and $x_2$ denote the vector of the last $g$ entries of $x$. Furthermore, for $x_i \in \mathbb{C}^g$, let $x_i^T$ denote the transpose of $x_i$. Then for $\xi \in (\frac{1}{2}\mathbb{Z})^{2g}$, we define
\begin{equation}
e_*(\xi) = \exp(4\pi i \xi_1^T  \xi_2)
\end{equation}
to be the \emph{parity} of $\xi$. Note that $e_*$ is also well-defined on $(\frac{1}{2}\mathbb{Z})^{2g}/\mathbb{Z}^{2g}$.
\end{definition}

\begin{proposition}\label{prop:poorsdescr}[Lemma 1.4.13 of \cite{poor}]
Let $X$ be a hyperelliptic curve of genus $g$ equipped with a marking $m$ of its branch points, and let $J(X)$ be equipped with a choice of small period matrix $\Omega$ via the process described in Definition \ref{def:periodmatrix}. Then the set $U(\Omega,m)$ of Definition \ref{def:Uset} is given by
\begin{equation}
\{i \in  \{1, 2, \ldots, 2g+1\} : e_*(\eta_{\Omega,m}(\{i, \infty\})) = -1 \} \cup \{\infty \}.
\end{equation}
\end{proposition}

In other words, if we consider the distinguished elements $D_i = P_i-P_{\infty} \in J(X)[2]$ for $i = 1, 2, \ldots, 2g+1, \infty$, the set $U(\Omega,m)$ can be made to contain $\infty$ as well as $i$ such that the coordinates of $D_i$ under the Abel-Jacobi map are odd, for $i = 1, 2, \ldots, 2g+1$.

\subsection{Mumford and Poor's vanishing theorem}\label{sec:vanishing}

We now turn our attention to explaining the significance of the set $U(\Omega,m)$. As we explained briefly in the introduction, the set connects the vanishing set of an analytic theta function to a distinguished divisor $\Theta$ on the Jacobian $J(X)$ of a marked hyperelliptic curve $X$.

We begin by defining this divisor:
\begin{definition}\label{def:Thetaset}
Let $X$ be a curve of genus $g$ defined over $\mathbb{C}$ and $P_\infty$ be a basepoint on $X$. Then we define the \emph{theta divisor} $\Theta$ on $J(X)$ to be the subset of divisor classes of the form
\begin{equation}
\sum_{i=1}^{g-1} Q_i - (g-1)P_{\infty}.
\end{equation} 
Note that if $X$ is a marked hyperelliptic curve and we choose $P_{\infty}$ to be the branch point of $X$ labeled $\infty$, this gives a unique choice of theta divisor on $J(X)$. We therefore call it ``the" theta divisor on the marked curve $X$.
\end{definition}

We now define the theta function whose zeroes we will study:

\begin{definition}\label{def:thetafunction}
For $z \in \mathbb{C}^g$ and $\Omega \in \mathbb{H}_g$, we define the \emph{theta function}
\begin{equation}
\vartheta(z, \Omega) = \sum_{n \in \mathbb{Z}^{g}}\exp(\pi i n^T \Omega n + 2 \pi i n^ T z).
\end{equation}
\end{definition}

\begin{remark}
As noted in the introduction, this function is quasi-periodic for the lattice $L_{\Omega}$ in the coordinate $z$. Indeed, if $k \in  \mathbb{Z}^{g}$, by \cite[p. 120]{mumford1}, we have
\begin{equation}
\vartheta(z+k,\Omega) = \vartheta(z,\Omega)
\end{equation}
and
\begin{equation}
\vartheta(z+\Omega k,\Omega) = \exp(-i\pi k^T\Omega k - 2\pi i k^Tz) \vartheta(z,\Omega).
\end{equation}
However, since the automorphy factor is non-zero, the zero set of $\vartheta$ is well-defined as a subset of $\mathbb{C}^g/L_{\Omega}$.
\end{remark}

For the convenience of the reader, we repeat the Riemann Vanishing Theorem now that all terms have been defined:
\begin{theorem}[Riemann Vanishing Theorem, or Theorem 5.3 of \cite{mumford2}]\label{thm:vanishing}
Let $X$ be a hyperelliptic curve, $m$ be a marking of its branch points, 
and let $\Omega$ be a small period matrix associated to its Jacobian via the process described in Definition \ref{def:periodmatrix}. If $\Theta \in J(X)$ is as in Definition \ref{def:Thetaset}, then  the zero set of the theta function $\vartheta(z,\Omega)$, considered as a subset of $J(X)(\mathbb{C})$ is a translate of $\Theta$ by a two-torsion point of $J(X)$.
\end{theorem}

From the introduction, we recall that this gives rise to the set $U(\Omega,m)$ in the following manner: Given a marking $m$ and a small period matrix $\Omega$, the Riemann Vanishing Theorem singles out a divisor on the Jacobian of $X$ (the zero locus of the function $\vartheta$). As this is a translate of $\Theta$ by a two-torsion point, this gives in turn a distinguished two-torsion point on $J(X)$. Recall that Proposition \ref{prop:GBisom} gives an isomorphism between the group $G_B$ defined in Proposition \ref{prop:defGB} and the two-torsion of $J(X)$. Therefore, via this isomorphism, we obtain an element of the group $G_B$. Finally, since the elements of $G_B$ are equivalence classes of certain subsets of $B$ (where the equivalence consists in taking the complement in $B = \{1,2, \ldots, 2g+1,\infty\}$), we obtain a certain (equivalence class of) subset of $B$, which we denote by $T(\Omega,m)$ here.

We then define the set $U(\Omega,m)$ to be the element of the equivalence class of
\begin{equation}
\begin{cases}
T(\Omega,m) & \text{if $g$ is odd, and} \\
T(\Omega,m) \circ \{\infty\} & \text{if $g$ is even}
\end{cases}
\end{equation}
that contains $\infty$, as noted in Definition \ref{def:Uset}.

This definition is motivated by the proof of Proposition 6.2 of \cite{mumford2}: Under the correspondence given in part a) of this Proposition, the set $T(\Omega,m)$ when $g$ is odd, or $T(\Omega,m) \circ \{\infty\}$ when $g$ is even, corresponds to the translate $\Theta + e_{T(\Omega,m)}$ and to the characteristic $\delta + \eta_{T(\Omega,m)}$ (in our notation $\eta_{T(\Omega,m)}$ is $\eta_{\Omega,m}(T(\Omega,m))$). Since $\eta_{T(\Omega,m)}=\delta$ and $\delta \in \frac{1}{2} L_{\Omega}$, $T(\Omega,m)$ when $g$ is odd, or $T(\Omega,m) \circ \{\infty\}$ when $g$ is even, corresponds to $0$ and is therefore the set $U(\Omega,m)$ defined here.

We end by giving part of the Vanishing Criterion for hyperelliptic small period matrices, which highlights how truly central the set $U(\Omega,m)$ is to the computational theory of hyperelliptic curves.
\begin{theorem}[Main Theorem 2.6.1 of \cite{poor}]\label{thm:vanishingtheorem}
Let $X$ be a hyperelliptic curve of genus $g$, with a marking of its branch points $m$ and let $\Omega$ be a small period matrix associated to its Jacobian $J(X)$ via the process described in Definition \ref{def:periodmatrix}. Then for $S \subseteq B$ with $ \#S\equiv 0 \pmod{2}$, we have
\begin{equation}\label{VanishingCondition}
\vartheta(AJ(e_S),\Omega) = 0 
\end{equation}
if and only if
\begin{equation}
 \#(S\circ U(\Omega,m)) \neq g+1.
 \end{equation}
\end{theorem}

We stress that here we have only stated part of the Vanishing Criterion for hyperelliptic matrices, and that the important part of this Vanishing Criterion for computational purposes is a strengthening of the statement which allows one to give a converse for general curves. This converse then allows the detection of hyperelliptic small period matrices among all small period matrices. We refer the reader to Poor's work \cite{poor}, notably Definition 1.4.11 for a complete account of this converse with proofs, or to \cite{BILV} for a shorter exposition.

\section{The proof}\label{sec:proof}

The proof of Theorem \ref{thm:main} has two main parts. In the first part, for a fixed $g \geq 1$ we count the number of different sets $U$ satisfying $U \subseteq  \{1, 2,\ldots, 2g+1, \infty\}$, $\infty \in U$ and $\# U \equiv g+1 \pmod{4}$ (this is Proposition \ref{prop:Usetcount}). In the second part, we count how many different sets $U(\Omega,m)$ arise as we vary among all possible small period matrices $\Omega$ that can be associated to the Jacobian of a hyperelliptic curve $X$ via the process described in Definition \ref{def:periodmatrix} and all possible markings $m$ of its branch points (this is Proposition \ref{prop:quotientcount}). Since these two numbers are equal, we conclude that every allowable set $U$ must arise $U(\Omega,m)$ for some choice of $\Omega$ and $m$.

\subsection{Counting the allowable sets $U$}

Counting the sets such that $U \subseteq  \{1, 2,\ldots, 2g+1, \infty\}$, $\# U \equiv g+1 \pmod{4}$, and $\infty \in U$ is equivalent to counting the sets satisfying the following two conditions:
\begin{itemize}
\item $\tilde{U} \subseteq  \{1, 2,\ldots, 2g+1\}$, and
\item $\# \tilde{U} \equiv g \pmod{4}$.
\end{itemize}
We turn to this task.

\begin{definition}
Let $n \geq 1$, $d \geq 0$ and $m \geq 2$ be integers. We define the sum
\begin{equation}
S(n,d,m) = \sum_{\substack{0 \leq k \leq n\\ k \equiv d \pmod{m}}} \binom{n}{k}.
\end{equation}
This is the number of subsets of $\{1, \ldots,n\}$ of any cardinality $k \equiv d \pmod{m}$.
\end{definition}

We are interested in computing the quantity $S(2g+1,g,4)$. We first note the following well-known result:
\begin{proposition}\label{prop:evenodd}
Let $n$ be any positive integer, then
\begin{equation}
S(n,0,2) = S(n,1,2) = 2^{n-1}.
\end{equation}
\end{proposition}

In other words, for any $n$, of the $2^n$ subsets of $\{1, \ldots,n\}$, half of them have even cardinality, and half have odd cardinality.

\begin{lemma}
We have
\begin{equation}
S(n,d,4) = S(n-1,d,4) + S(n-1,d-1,4).
\end{equation}
\end{lemma}

\begin{proof}
This follows from Pascal's identity, which says that for $n \geq 1$ and $k \geq 0$, we have 
\begin{equation}
\binom{n}{k} = \binom{n-1}{k}+\binom{n-1}{k-1}.
\end{equation}
Here we use the usual convention that $\binom{n}{k} =0$ if $k < 0$.
\end{proof}

This is enough to show

\begin{proposition}\label{prop:Usetcount}
Let $g \geq 1$, then
\begin{equation}
S(2g+1,g,4) = 2^{g-1}(2^g+1).
\end{equation}
\end{proposition}

\begin{proof}
The proof is done by induction on $g$. The case of $g=1$ is the claim that $S(3,1,4) = 3$. Indeed, of the subsets of $\{1,2,3\}$, three of them have cardinality congruent to $1$ modulo $4$ (and therefore actually equal to $1$, since there are no subsets of $\{1,2,3\}$ of cardinality greater than or equal to $5$).

We know assume that $S(2g-1,g-1,4) = 2^{g-2}(2^{g-1}+1)$ and $g \geq 2$. We have
\begin{align}
S(2g+1,g,4) & = S(2g,g,4) +S(2g,g-1,4) \\
& = (S(2g-1,g,4) + S(2g-1,g-1,4)) \\
& \qquad + (S(2g-1,g-1,4) + S(2g-1,g-2,4)) \notag \\ 
& = S(2g-1,g,4) + S(2g-1,g-2,4) \\ 
&\qquad + 2S(2g-1,g-1,4). \notag
\end{align}

We now note that if $g$ is even, then
\begin{equation}
S(2g-1,g,4) + S(2g-1,g-2,4) = S(2g-1,0,2),
\end{equation}
and if $g$ is odd, then
\begin{equation}
S(2g-1,g,4) + S(2g-1,g-2,4) = S(2g-1,1,2).
\end{equation}
In either case, by Proposition \ref{prop:evenodd},
\begin{equation}
S(2g-1,g,4) + S(2g-1,g-2,4) = 2^{2g-2}.
\end{equation}
Furthermore, by induction $S(2g-1,g-1,4) = 2^{g-2}(2^{g-1}+1)$.

Therefore we have
\begin{align}
S(2g+1,g,4) & = 2^{2g-2} + 2 \cdot 2^{g-2}(2^{g-1}+1) \\
& = 2^{g-1} (2^{g-1} +2^{g-1} + 1) \\
& = 2^{g-1} (2^g+1).
\end{align}
 This completes the proof.
\end{proof}

\subsection{Counting the different sets $U(\Omega,m)$ for a hyperelliptic curve}

Here we show that in fact, given a hyperelliptic curve $X$ with a marking $m$ of its branch points, every allowable $U$-set is realized as $U(\Omega,m)$ as we vary the small period matrix $\Omega$ associated to its Jacobian $J(X)$ by Definition \ref{def:periodmatrix}. This certainly implies our main theorem. Thus we begin by fixing a marking $m$ on the branch points of $X$.

The proof is carried out by considering the action of $\Sp_{2g}(\mathbb{Z})$ on $\Omega$ and considering which matrices fix the set $U(\Omega,m)$. We will see in Proposition \ref{prop:Gamma12} that they are exactly a subgroup of $\Sp_{2g}(\mathbb{Z})$ denoted $\Gamma_{1,2}$:

\begin{definition}
Let $\Gamma_{1,2}$ be the subgroup of $\Sp_{2g}(\mathbb{Z})$ containing the matrices that fix the parity of every element of $(\frac{1}{2}\mathbb{Z})^{2g}$. In other words, $\gamma \in \Gamma_{1,2}$ if and only if
\begin{equation}
e_*(\gamma \xi) = e_*(\xi)
\end{equation}
for all $\xi \in (\frac{1}{2}\mathbb{Z})^{2g}$, where $e_*$ is as in Definition \ref{def:estar} and $\gamma \xi$ is the usual matrix-vector multiplication. 
\end{definition}

We will need two further characterizations of these matrices below. First, we have:

\begin{proposition}\label{prop:gamma12char1}
Let $\gamma \in \Sp_{2g}(\mathbb{Z})$ with 
\begin{equation}
\gamma = \left(\begin{matrix} A & B \\ C & D \end{matrix}\right)
\end{equation}
where $A$, $B$, $C$ and $D$ are four $g \times g$ matrices. 
Then $\gamma \in \Gamma_{1,2}$ if and only if the diagonals of the matrices $A^TC$ and $B^TD$ have all even entries.
\end{proposition}

\begin{proof}
This can be verified directly, or found in \cite[page 189]{mumford1}.
\end{proof}

The second characterization of these matrices relies on 
an important property of the vectors $\eta_{\Omega,m}(\{i, \infty\})$ for $i = 1, 2, \ldots, 2g+1$: 

\begin{proposition}\label{prop:asygetic}
Let $X$ be a marked hyperelliptic curve, $J(X)$ its Jacobian, and $\Omega$ a small period matrix associated to $J(X)$ via the process outlined in Definition \ref{def:periodmatrix}. Furthermore, given this data, let $\eta_{\Omega,m}$ be the map given in Definition \ref{def:etamap}. Then the set
\begin{equation}
\{\eta_{\Omega,m}(\{i, \infty\}) :  i = 1, \ldots, 2g+1\}
\end{equation}
contains a basis of the $\mathbb{F}_2$-vector space $(\frac{1}{2}\mathbb{Z})^{2g}/\mathbb{Z}^{2g}$.
\end{proposition}

\begin{proof}
By the proof Lemma 1.4.13 of \cite{poor}, the set 
\begin{equation}
\{\eta_{\Omega,m}(\{i, \infty\}) :  i = 1, \ldots, 2g+1\}
\end{equation}
is an azygetic basis of $(\frac{1}{2}\mathbb{Z})^{2g}/\mathbb{Z}^{2g}$, and by Definition 1.4.12 of \emph{ibid}, therefore spans the vector space $(\frac{1}{2}\mathbb{Z})^{2g}/\mathbb{Z}^{2g}$. Therefore it contains a basis of the space.
\end{proof}

We can now prove the following:

\begin{lemma}\label{lem:Gamma12}
A matrix $\gamma \in \Sp_{2g}(\mathbb{Z})$ belongs to $\Gamma_{1,2}$ if and only if it fixes the parity of $\eta_{\Omega,m}(\{i, \infty\})$ for $i = 1, 2, \ldots, 2g+1$.
\end{lemma}

\begin{proof}
It is clear that if $\gamma \in \Gamma_{1,2}$, then it will fix the parity of $\eta_{\Omega,m}(\{i, \infty\})$ for $i = 1, 2, \ldots, 2g+1$. Therefore we assume that $\gamma \in \Sp_{2g}(\mathbb{Z})$ fixes the parity of $\eta_{\Omega,m}(\{i, \infty\})$ for $i = 1, 2, \ldots, 2g+1$ and show that $\gamma \in \Gamma_{1,2}$.

We first establish some notation: For $\xi \in (\frac{1}{2}\mathbb{Z})^{2g}$, let
\begin{equation}
q(\xi) = \xi_1^T\xi_2
\end{equation}
be the quadratic form associated to the parity function $e_*$ defined in Definition \ref{def:estar}. We note that 
\begin{equation}
q(\xi) \equiv q(\zeta) \pmod{(\frac{1}{2}\mathbb{Z})^{2g}},
\end{equation}
if and only if
\begin{equation}
e_*(\xi) = e_*(\zeta).
\end{equation}

Let also 
\begin{equation}
b(\xi,\zeta) = \xi^T J \zeta,
\end{equation}
be the bilinear form associated to the matrix $J$, where as before 
\begin{equation}
J = \left(\begin{smallmatrix} 0 & \mathds{1}_g \\  -\mathds{1}_g & 0 \end{smallmatrix}\right),
\end{equation}
and $\mathds{1}_g$ is the $g \times g$ identity matrix. 

A quick computation shows that for any $\xi,\zeta \in (\frac{1}{2}\mathbb{Z})^{2g}$
\begin{equation}
q(\xi+\zeta) \equiv  q(\xi) + q(\zeta) + b(\xi, \zeta) \pmod{(\frac{1}{2}\mathbb{Z})^{2g}}.
\end{equation}

Now let $\gamma \in \Sp_{2g}(\mathbb{Z})$. We have then that 
\begin{equation}
b(\gamma \xi, \gamma \zeta ) = b(\xi, \zeta),
\end{equation}
by definition of $b$ and $\Sp_{2g}(\mathbb{Z})$. Therefore, for any $\xi,\zeta \in (\frac{1}{2}\mathbb{Z})^{2g}$
\begin{align}
q(\gamma(\xi+\zeta)) = q(\gamma\xi + \gamma\zeta) & \equiv  q(\gamma\xi) + q(\gamma\zeta) + b(\gamma\xi, \gamma\zeta) \pmod{(\frac{1}{2}\mathbb{Z})^{2g}} \\
& \equiv  q(\gamma\xi) + q(\gamma\zeta) + b(\xi, \zeta) \pmod{(\frac{1}{2}\mathbb{Z})^{2g}}.
\end{align}
As a result, if
\begin{equation}
q(\gamma\xi) \equiv q(\xi) \pmod{(\frac{1}{2}\mathbb{Z})^{2g}}
\end{equation}
and
\begin{equation}
q(\gamma\zeta) \equiv q(\zeta) \pmod{(\frac{1}{2}\mathbb{Z})^{2g}},
\end{equation}
then
\begin{equation}
q(\gamma(\xi+\zeta)) \equiv q(\xi+ \zeta) \pmod{(\frac{1}{2}\mathbb{Z})^{2g}}.
\end{equation}

From this discussion we conclude that if $e_*(\gamma\xi) = e_*(\xi)$ and $e_*(\gamma\zeta) = e_*(\zeta)$, it follows that
\begin{equation}
e_*(\gamma(\xi+\zeta)) = e_*(\xi+\zeta).
\end{equation}

The result now follows from the fact that the set
\begin{equation}
\{\eta_{\Omega,m}(\{i, \infty\}) :  i = 1, \ldots, 2g+1\}
\end{equation}
contains a basis of the $\mathbb{F}_2$-vector space $(\frac{1}{2}\mathbb{Z})^{2g}/\mathbb{Z}^{2g}$ by Proposition \ref{prop:asygetic}. Therefore if a matrix $\gamma$ fixes the parity of each element of this basis, it must fix the parity of each element of the vector space.
\end{proof}

We are now in a position to show:

\begin{proposition}\label{prop:Gamma12}
Let $X$ be a marked hyperelliptic curve, $J(X)$ its Jacobian, and $\Omega$ a small period matrix associated to $J(X)$ via the process outlined in Definition \ref{def:periodmatrix}. Furthermore, given this data, let $\eta_{\Omega,m}$ be the map given in Definition \ref{def:etamap} and $U(\Omega,m)$ be the set defined in Definition \ref{def:Uset}.

Let $\gamma \in \Sp_{2g}(\mathbb{Z})$. Then the matrix $\gamma \cdot \Omega$ is another small period matrix for $J(X)$, to which we may similarly attach a map $\eta_{\gamma \cdot \Omega,m}$ and a set $U(\gamma \cdot \Omega,m)$.

In that case, we have
\begin{equation}
U(\gamma \cdot \Omega,m) = U(\Omega,m)
\end{equation}
if and only if
\begin{equation}
\gamma \in \Gamma_{1,2}.
\end{equation}
\end{proposition}

\begin{proof}
Recall from Proposition \ref{prop:poorsdescr} that $U(\Omega,m)$ can be described as the set
\begin{equation}
\{i \in  \{1, 2, \ldots, 2g+1\} : e_*(\eta_{\Omega,m}(\{i, \infty\})) = -1 \} \cup \{\infty \}.
\end{equation}
Since $\eta_{\Omega,m}(\{i, \infty\}) \in (\frac{1}{2}\mathbb{Z})^{2g}/\mathbb{Z}^{2g}$ is none other than $AJ_c(e_{\{i, \infty\}})$, by Proposition \ref{prop:sympaction}, we have
\begin{equation}
\eta_{\gamma \cdot \Omega,m}(\{i, \infty\}) = \gamma^{-T} \eta_{\Omega,m}(\{i, \infty\}),
\end{equation}

Therefore we have that
\begin{equation}
U(\gamma \cdot \Omega,m) = U(\Omega,m)
\end{equation}
if and only if multiplication by $\gamma^{-T}$ does not change the parity of any $\eta_{\Omega,m}(\{i, \infty\})$ for $i = 1, 2, \ldots, 2g+1$. By Lemma \ref{lem:Gamma12}, this is the case if and only if $\gamma^{-T} \in \Gamma_{1,2}$.

To finish the proof we must show that $\gamma^{-T} \in \Gamma_{1,2}$ if and only if $\gamma \in \Gamma_{1,2}$. Note that since $\gamma \in \Sp_{2g}(\mathbb{Z})$, we have
\begin{equation}
\gamma^{-T} = \left(\begin{matrix} D & -C \\ -B & A \end{matrix} \right).
\end{equation}
By Proposition \ref{prop:gamma12char1}, it suffices thus to show that the diagonals of the matrices
\begin{equation}
D^T(-B) = (-B^T D)^T
\end{equation}
and
\begin{equation}
(-C)^T A = (-A^T C)^T
\end{equation}
have all even entries if and only if the diagonals of the matrices $A^TC$ and $B^TD$ have all even entries, which is true.
\end{proof}

As a direct consequence we now have:

\begin{theorem}
The number of different sets $U(\Omega,m)$ that arise, as $\Omega$ varies over all small period matrices that can be attached to the polarized Jacobian of a marked hyperelliptic curve $X$ with the process outlined in Definition \ref{def:periodmatrix}, is equal to the cardinality of the quotient group
\begin{equation}
\Sp_{2g}(\mathbb{Z})/\Gamma_{1,2}.
\end{equation}
\end{theorem}

\begin{proof}
As described in Section \ref{sec:preliminaries}, the group $\Sp_{2g}(\mathbb{Z})$ acts transitively on the set of small period matrices that can be associated to $J(X)$ via the process described in Definition \ref{def:periodmatrix}. This action changes $U(\Omega,m)$ if and only if $\gamma \in \Gamma_{1,2}$ by Proposition \ref{prop:Gamma12}, which completes the proof.
\end{proof}

We now compute the cardinality of this quotient group, which will give us the number of different sets $U(\Omega,m)$ attached to a fixed hyperelliptic curve $X$ with a marking of its branch points $m$ as $\Omega$ is allowed to vary over all possible small period matrices that can be associated to its Jacobian $J(X)$ via the process described in Definition \ref{def:periodmatrix}.

\begin{proposition}\label{prop:quotientcount}
We have that
\begin{equation}
\# \Sp_{2g}(\mathbb{Z})/\Gamma_{1,2} = 2^{g-1}(2^g+1).
\end{equation}
\end{proposition}

\begin{proof}
To compute the cardinality of this quotient group, we use the following facts: First, by the Third Group Isomorphism Theorem,
\begin{equation}
\Sp_{2g}(\mathbb{Z})/\Gamma_{1,2} \cong \frac{\Sp_{2g}(\mathbb{Z})/\Gamma(2)}{\Gamma_{1,2}/\Gamma(2)},
\end{equation}
where 
\begin{equation}
\Gamma(2) = \left\{ \gamma \in \Sp_{2g}(\mathbb{Z}) : \gamma \equiv \mathds{1}_{2g} \pmod{2}\right\}.
\end{equation}
Furthermore, we have
\begin{equation}
\Sp_{2g}(\mathbb{Z})/\Gamma(2) \cong \Sp_{2g}(\mathbb{F}_2),
\end{equation}
where $\Sp_{2g}(\mathbb{F}_2)$ is the group of matrices with coefficients in $\mathbb{F}_2$ and symplectic with respect to the bilinear form given by the matrix
\begin{equation}
\begin{pmatrix}
0 & \mathds{1}_g \\ \mathds{1}_g & 0
\end{pmatrix},
\end{equation}
and
\begin{equation}
\Gamma_{1,2}/\Gamma(2) \cong \SO_{2g}(\mathbb{F}_2,+1),
\end{equation}
where $\SO_{2g}(\mathbb{F}_2,+1)$ is the special orthogonal group of matrices with entries in $\mathbb{F}_2$ and preserving the quadratic form 
\begin{equation}
Q(x_1,x_2, \ldots, x_{2g-1},x_{2g}) = \sum_{i=1}^{g} x_i x_{g+i}.
\end{equation}
(These last two facts are implicit in the discussion in \cite{mumford1}, Appendix to Chapter 5.) 

There are therefore
\begin{equation}
\frac{\#  \Sp_{2g}(\mathbb{F}_2)}{\# \SO_{2g}(\mathbb{F}_2,+1)}
\end{equation}
different sets $U(\Omega,m)$ as $\Omega$ varies over all small period matrices that can be associated to $J(X)$ via the process described in Definition \ref{def:periodmatrix}. 

We have
\begin{equation}
\#  \Sp_{2g}(\mathbb{F}_2) = 2^{g^2} \prod_{i=1}^{g}(2^{2i}-1),
\end{equation}
(see for example \cite[Theorem 3.12]{grove}) and
\begin{equation}
\# \SO_{2g}(\mathbb{F}_2,+1) = 2 \cdot 2^{g(g-1)}(2^g-1)\prod_{i=1}^{g-1}(2^{2i}-1),
\end{equation}
(see for example \cite[Table 2.1C]{kleidman}).
Computing the quotient gives the result we sought.

\end{proof}

\bibliography{bibliography} {}

\begin{thebibliography}{Mum07b}

\bibitem[BILV16]{BILV}
Jennifer~S. Balakrishnan, Sorina Ionica, Kristin Lauter, and Christelle
  Vincent.
\newblock Constructing genus-3 hyperelliptic {J}acobians with {CM}.
\newblock {\em London Mathematical Society Journal of Computation and
  Mathematics}, 19(suppl. A):283--300, 2016.

\bibitem[BL04]{Birkenhake}
Christina Birkenhake and Herbert Lange.
\newblock {\em Complex abelian varieties}, volume 302 of {\em Grundlehren der
  Mathematischen Wissenschaften}.
\newblock Springer-Verlag, Berlin, second edition, 2004.

\bibitem[Gro02]{grove}
Larry~C. Grove.
\newblock {\em Classical groups and geometric algebra}, volume~39 of {\em
  Graduate Studies in Mathematics}.
\newblock American Mathematical Society, Providence, RI, 2002.

\bibitem[KL90]{kleidman}
Peter Kleidman and Martin Liebeck.
\newblock {\em The subgroup structure of the finite classical groups}, volume
  129 of {\em London Mathematical Society Lecture Note Series}.
\newblock Cambridge University Press, 1990.

\bibitem[Mum07a]{mumford1}
David Mumford.
\newblock {\em Tata lectures on theta. {I}}.
\newblock Modern Birkh{\"a}user Classics. Birkh{\"a}user, Boston, MA, 2007.

\bibitem[Mum07b]{mumford2}
David Mumford.
\newblock {\em Tata lectures on theta. {II}}.
\newblock Modern Birkh{\"a}user Classics. Birkh{\"a}user, Boston, MA, 2007.

\bibitem[Poo94]{poor}
Cris Poor.
\newblock The hyperelliptic locus.
\newblock {\em Duke Mathematical Journal}, 76(3):809--884, 1994.

\end{thebibliography}
\bibliographystyle{alpha}

\end{document}